\documentclass{article}%
\usepackage{amssymb,latexsym}
\usepackage{amsfonts}
\usepackage{amsmath}
\usepackage[colorlinks=true, pdfstartview=FitV, linkcolor=blue,
citecolor=blue, urlcolor=blue]{hyperref}
\usepackage{color}
\usepackage{amssymb}
\usepackage{graphicx}%

\newtheorem{theorem}{Theorem}

\newtheorem{corollary}[theorem]{Corollary}

\newtheorem{example}[theorem]{Example}

\newtheorem{remark}[theorem]{Remark}

\newenvironment{proof}[1][Proof]{\noindent\textbf{#1.} }{\ \rule{0.5em}{0.5em}}
\numberwithin{theorem}{section}
\begin{document}

\title{On the integral of products of higher-order Bernoulli and Euler polynomials}
\author{M. Cihat Da\u{g}l\i\ and M\"{u}m\"{u}n Can\\
Department of Mathematics, Akdeniz University, Antalya, TR-07058, Turkey\\
\textbf{E-mails:} mcihatdagli@akdeniz.edu.tr, mcan@akdeniz.edu.tr}
\date{}
\maketitle

\begin{abstract}In this paper, we derive a formula on the
integral of products of the higher-order Euler polynomials. By the same way,
similar relations are obtained for $l$ higher-order Bernoulli polynomials and
$r$ higher-order Euler polynomials. Moreover, we establish the connection
between the results and the generalized Dedekind sums and Hardy--Berndt sums.
Finally, the Laplace transform of Euler polynomials is given.

\textbf{Keywords:} Bernoulli polynomials and numbers,
Dedekind sums, Integrals, Recurrence relations.

\textbf{2010 Mathematics Subject Classification:} 11B68, 11F20.
\end{abstract}
\section{Introduction}

The classical Bernoulli polynomials $B_{m}(x)$ and Euler polynomials
$E_{m}(x)$ are usually defined by means of the following generating functions:%
\begin{equation}
\dfrac{ue^{uz}}{e^{u}-1}=%
{\displaystyle\sum\limits_{m=0}^{\infty}}
B_{m}(z)\dfrac{u^{m}}{m!}\text{ }\left(  \left\vert u\right\vert <2\pi\right)
\text{ and }\dfrac{2e^{uz}}{e^{u}+1}=%
{\displaystyle\sum\limits_{m=0}^{\infty}}
E_{m}(z)\dfrac{u^{m}}{m!}\text{ }\left(  \left\vert u\right\vert <2\pi\right)
. \label{0}%
\end{equation}
In particular, the rational numbers $B_{m}=B_{m}(0)$ and integers $E_{m}%
=2^{m}E_{m}(1/2)$ are called classical Bernoulli numbers and Euler numbers, respectively.

As is well known, the classical Bernoulli and Euler polynomials play important
roles in different areas of mathematics such as number theory, combinatorics,
special functions and analysis.

This paper is primarily concerned with the higher-order Bernoulli and Euler
polynomials. We derive a formula for the integral having $r$ higher-order
Euler polynomials and also for $l$ higher-order Bernoulli and $r$ higher-order
Euler polynomials. The result is the corresponding generalization of some
formulae discovered by Agoh and Dilcher \cite{ad}, Hu et al \cite{hkk} and of
course \cite{c,em,mi,mo,n,w}. From our formula, we establish the connection
between the sums of products of Euler (and Bernoulli and Euler) polynomials
and the reciprocity formula for generalized Dedekind (and Hardy--Berndt) sums,
motivated by Da\u{g}l\i\ and Can \cite{cm}.

We now turn to the higher-order Bernoulli and higher-order Euler polynomials.
The higher-order Bernoulli polynomials $B_{m}^{(\alpha)}(x)$ and higher-order
Euler polynomials $E_{m}^{(\alpha)}(x)$, each of degree $m$ in $x$ and in
$\alpha$, are defined by means of the generating functions \cite{n}
\[
\left(  \dfrac{u}{e^{u}-1}\right)  ^{\alpha}e^{uz}=%
{\displaystyle\sum\limits_{m=0}^{\infty}}
B_{m}^{(\alpha)}(z)\dfrac{u^{m}}{m!}\text{ and }\left(  \dfrac{2}{e^{u}%
+1}\right)  ^{\alpha}e^{uz}=%
{\displaystyle\sum\limits_{m=0}^{\infty}}
E_{m}^{(\alpha)}(z)\dfrac{u^{m}}{m!},
\]
respectively. For $\alpha=1$, we have $B_{m}^{(1)}(z)=B_{m}(z)$ and
$E_{m}^{(1)}(z)=E_{m}(z)$. They possess the differential property%
\begin{equation}
\dfrac{d}{dz}B_{m}^{(\alpha)}(z)=mB_{m-1}^{(\alpha)}(z),\text{ }\dfrac{d}%
{dz}E_{m}^{(\alpha)}(z)=mE_{m-1}^{(\alpha)}(z) \label{1}%
\end{equation}
and reciprocal relations%
\begin{equation}
B_{m}^{(\alpha)}(\alpha-z)=\left(  -1\right)  ^{m}B_{m}^{(\alpha)}(z),\text{
}E_{m}^{(\alpha)}(\alpha-z)=\left(  -1\right)  ^{m}E_{m}^{(\alpha)}(z)
\label{2}%
\end{equation}
which imply $B_{m}^{(\alpha)}(\alpha/2)=0$ and $E_{m}^{(\alpha)}(\alpha/2)=0$
for odd $m.$

Also, we need the following expression of the Euler polynomials in terms of
Bernoulli polynomials%
\begin{equation}
E_{n}(x)=\frac{2}{n+1}\left\{  B_{n+1}(x)-2^{n+1}B_{n+1}\left(  x/2\right)
\right\}  \label{49}%
\end{equation}
for $n\geq0.$

We summarize this study as follows: we firstly obtain several convolution
formulas for higher-order Bernoulli and Euler polynomials applying the
generating function methods, motivated by \cite{cz}. We also derive a formula
for the integral having higher-order Euler polynomials. By this, we extend the
result of Hu et al \cite{hkk} and Liu et al \cite{lpz}. By the same way,
similar relations are obtained for $l$ higher-order Bernoulli polynomials and
$r$ higher-order Euler polynomials, as well\textbf{.} Moreover, we establish
the connection between the results and the reciprocity formulas for
generalized Dedekind sums $T_{r}(c,d)$ and generalized Hardy-Berndt sums
$s_{3,r}(c,d)$ and $s_{4,r}(c,d)$.

\section{Convolutions of higher-order Bernoulli and Euler polynomials}

In this section, we obtain some convolutions involving higher-order Bernoulli
and Euler polynomials we will use in the next section.

Differentiating the generating function of higher-order Euler polynomials as follows%

\begin{align*}
\frac{d}{du}\left(  \left(  \frac{2}{e^{u}+1}\right)  ^{n}e^{uz}\right)   &
=\frac{2^{n}ze^{uz}}{\left(  e^{u}+1\right)  ^{n}}-\frac{n2^{n}e^{u(z+1)}%
}{\left(  e^{u}+1\right)  ^{n+1}}\\
&  =2^{n}\frac{d}{du}\frac{e^{uz}}{\left(  e^{u}+1\right)  ^{n}},
\end{align*}
we have%
\[
\frac{n2^{n+1}e^{u(z+1)}}{\left(  e^{u}+1\right)  ^{n+1}}=\frac{2^{n+1}%
ze^{uz}}{\left(  e^{u}+1\right)  ^{n}}-2^{n+1}\frac{d}{du}\frac{e^{uz}%
}{\left(  e^{u}+1\right)  ^{n}}.
\]
Taking $z=x+y-1$ and $n=\beta+\gamma-1$ leads%
\begin{align*}
\frac{n2^{n+1}e^{u(z+1)}}{\left(  e^{u}+1\right)  ^{n+1}}  &  =\left(
{\displaystyle\sum\limits_{m=0}^{\infty}}
E_{m}^{(\beta)}(x)\frac{u^{m}}{m!}\right)  \left(
{\displaystyle\sum\limits_{k=0}^{\infty}}
E_{k}^{(\gamma)}(y)\frac{u^{k}}{k!}\right) \\
&  =%
{\displaystyle\sum\limits_{m=0}^{\infty}}
{\displaystyle\sum\limits_{k=0}^{m}}
\binom{m}{k}E_{k}^{(\beta)}(x)E_{m-k}^{(\gamma)}(y)\frac{u^{m}}{m!},
\end{align*}%
\[
\frac{2^{n+1}ze^{uz}}{\left(  e^{u}+1\right)  ^{n}}=2\left(  x+y-1\right)
{\displaystyle\sum\limits_{m=0}^{\infty}}
E_{m}^{(\beta+\gamma-1)}(x+y-1)\frac{u^{m}}{m!}%
\]
and%
\[
2^{n+1}\frac{d}{du}\frac{e^{uz}}{\left(  e^{u}+1\right)  ^{n}}=2%
{\displaystyle\sum\limits_{m=0}^{\infty}}
E_{m+1}^{(\beta+\gamma-1)}(x+y-1)\frac{u^{m}}{m!}.
\]
By equating the coefficients of $\dfrac{u^{m}}{m!},$ we get the convolution
formula%
\begin{equation}%
{\displaystyle\sum\limits_{k=0}^{m}}
\binom{m}{k}E_{k}^{(\beta)}(x)E_{m-k}^{(\gamma)}(y)=2\left(  x+y-1\right)
E_{m}^{(\beta+\gamma-1)}(x+y-1)-2E_{m+1}^{(\beta+\gamma-1)}(x+y-1). \label{41}%
\end{equation}
Similarly, for higher-order Bernoulli polynomials, we obtain
\begin{align*}
&
{\displaystyle\sum\limits_{k=0}^{m}}
\binom{m}{k}B_{k}^{(\beta)}(x)B_{m-k}^{(\gamma)}(y)\\
&  =\left(  x+y-1\right)  mB_{m-1}^{(\beta+\gamma-1)}(x+y-1)+\left(
\gamma+\beta-1-m\right)  B_{m}^{(\beta+\gamma-1)}(x+y-1).
\end{align*}
From the generating functions of the higher-order Bernoulli and Euler
polynomials, we can write
\[
\left(  \frac{u}{e^{u}-1}\right)  ^{n}e^{xu}\left(  \frac{2}{e^{u}+1}\right)
^{n}e^{yu}=\left(  \frac{2u}{e^{2u}-1}\right)  ^{n}e^{u(x+y)}.
\]
Thus, similar arguments give the following convolution formula
\begin{equation}%
{\displaystyle\sum\limits_{k=0}^{m}}
\binom{m}{k}B_{m-k}^{\left(  n\right)  }(x)E_{k}^{\left(  n\right)  }%
(y)=2^{m}B_{m}^{(n)}\left(  \frac{x+y}{2}\right)  . \label{34}%
\end{equation}

\section{\textbf{Integral of products of higher-order Bernoulli and Euler
polynomials}}

This section is devoted to obtain the integral of products of $r$ higher-order
Euler polynomials. Also, we derive a formula for the integral of products of
$l$ higher-order Bernoulli polynomials and $r$ higher-order Euler polynomials.
Furthermore, we relate these results to the reciprocity formulas for
generalized Dedekind sums $T_{r}(c,d)$ and Hardy-Berndt sums $s_{3,r}(c,d)$
and $s_{4,r}(c,d)$.

\subsection{\textbf{\textit{Euler polynomials}}}

\begin{theorem}
\label{th-i}Let $b_{1},...,b_{r},$ $y_{1},...,y_{r}$ be arbitrary real numbers
with $b_{s}\not =0,$ $1\leq s\leq r,$ and
\begin{align*}
\widehat{I}_{n_{1},...,n_{r}}(x;b;y)  &  =\widehat{I}_{n_{1},...,n_{r}%
}(x;b_{1},...,b_{r};y_{1},...,y_{r})\\
&  =\frac{1}{n_{1}!\cdots n_{r}!}\int\limits_{0}^{x}\prod\limits_{s=1}%
^{r}E_{n_{s}}^{\left(  \alpha_{s}\right)  }\left(  b_{s}z+y_{s}\right)  dz,\\
\widehat{C}_{n_{1},...,n_{r}}(x;b;y)  &  =\widehat{C}_{n_{1},...,n_{r}}\left(
x;b_{1},...,b_{r};y_{1},...,y_{r}\right) \\
&  =\frac{1}{n_{1}!\cdots n_{r}!}\left(  \prod\limits_{s=1}^{r}E_{n_{s}%
}^{\left(  \alpha_{s}\right)  }\left(  b_{s}x+y_{s}\right)  -\prod
\limits_{s=1}^{r}E_{n_{s}}^{(\alpha_{s})}\left(  y_{s}\right)  \right)
\end{align*}
Then%
\begin{align*}
&  \widehat{I}_{n_{1},...,n_{r}}(x;b;y)\\
&  =\sum\limits_{a=0}^{\mu}\left(  -1\right)  ^{a}\sum\limits_{j_{1}%
+\cdots+j_{r-1}=a}\binom{a}{j_{1},...,j_{r-1}}b_{1}^{j_{1}}\cdots
b_{r-1}^{j_{r-1}}b_{r}^{-a-1}\widehat{C}_{n_{1}-j_{1},\ldots,n_{r-1}%
-j_{r-1},n_{r}+a+1}(x;b;y)\\
&  \quad+\frac{\left(  -1\right)  ^{\mu+1}}{\left(  n+\mu+1\right)  !}%
\int\limits_{0}^{x}\left(  \prod\limits_{s=1}^{r-1}E_{n_{s}}^{\left(
\alpha_{s}\right)  }\left(  b_{s}z+y_{s}\right)  \right)  ^{\left(
\mu+1\right)  }E_{n_{r}+\mu+1}^{(\alpha_{r})}\left(  b_{r}z+y_{r}\right)  dz,
\end{align*}
where $\binom{\mu}{n_{1},...,n_{r}}$ are the multinomial coefficients defined
by%
\[
\binom{\mu}{n_{1},...,n_{r}}=\frac{\mu!}{n_{1}!\cdots n_{r}!},\text{ }%
n_{1}+\cdots+n_{r}=\mu\text{ and }n_{1},...,n_{r}\geq0.
\]
In particular if $\mu=n_{1}+\cdots+n_{r-1},$ we have%
\begin{align}
\widehat{I}_{n_{1},...,n_{r}}(x;b;y)  &  =\sum\limits_{a=0}^{\mu}\left(
-1\right)  ^{a}\sum\limits_{j_{1}+\cdots+j_{r-1}=a}\binom{a}{j_{1}%
,...,j_{r-1}}b_{1}^{j_{1}}\cdots b_{r-1}^{j_{r-1}}\nonumber\\
&  \quad\times b_{r}^{-a-1}\widehat{C}_{n_{1}-j_{1},\ldots,n_{r-1}%
-j_{r-1},n_{r}+a+1}(x;b;y). \label{31}%
\end{align}

\end{theorem}

\begin{proof}
Let
\[
f(z)=E_{n_{1}}^{(\alpha_{1})}\left(  b_{1}z+y_{1}\right)  \cdots E_{n_{r-1}%
}^{(\alpha_{r-1})}\left(  b_{r-1}z+y_{r-1}\right)  .
\]
Then%
\begin{align*}
&  \frac{1}{n_{r}!}\int\limits_{0}^{x}f(z)E_{n_{r}}^{(\alpha_{r})}\left(
b_{r}z+y_{r}\right)  dz\\
&  =\left[  \frac{1}{b_{r}\left(  n_{r}+1\right)  !}f(z)E_{n_{r}+1}%
^{(\alpha_{r})}\left(  b_{r}z+y_{r}\right)  \right]  _{0}^{x}-\frac{1}{\left(
n_{r}+1\right)  !}\int\limits_{0}^{x}f^{\prime}(z)E_{n_{r+1}}^{(\alpha_{r}%
)}\left(  b_{r}z+y_{r}\right)  dz.
\end{align*}
Using $\mu$ additional integrations by parts, we find that%
\begin{align}
\frac{1}{n_{r}!}\int\limits_{0}^{x}f(z)E_{n_{r}}^{(\alpha_{r})}\left(
b_{r}z+y_{r}\right)  dz  &  =\sum\limits_{a=0}^{\mu}\frac{\left(  -1\right)
^{a}}{\left(  n_{r}+a+1\right)  !}\left[  f^{(a)}(z)E_{n_{r}+a+1}^{(\alpha
_{r})}\left(  b_{r}z+y_{r}\right)  \right]  _{0}^{x}\nonumber\\
&  +\frac{\left(  -1\right)  ^{\mu+1}}{\left(  n_{r}+\mu+1\right)  !}%
\int\limits_{0}^{x}f^{(\mu+1)}(z)E_{n_{r}+\mu+1}^{(\alpha_{r})}\left(
b_{r}z+y_{r}\right)  dz. \label{40}%
\end{align}
Using the property of derivative%
\[
\left(  f_{1}\left(  z\right)  \cdots f_{m}\left(  z\right)  \right)
^{\left(  a\right)  }=\sum\limits_{j_{1}+\cdots+j_{m}=a}\binom{a}%
{j_{1},...,j_{m}}f_{1}{}^{\left(  j_{1}\right)  }\left(  z\right)  \cdots
f_{m}^{\left(  j_{m}\right)  }\left(  z\right)  ,
\]
and (\ref{1}), we get the desired result.
\end{proof}

Setting $x=1$ and $b_{s}=\alpha_{s}-2y_{s}$ with $y_{s}\not =\alpha_{s}/2,$
$1\leq s\leq r,$ in (\ref{31}) we have%
\begin{align*}
&  \widehat{I}_{n_{1},...,n_{r}}\left(  1;\alpha_{1}-2y_{1},\ldots,\alpha
_{r}-2y_{r};y_{1},...,y_{r}\right) \\
&  =\sum\limits_{a=0}^{n_{1}+\cdots+n_{r-1}}\left(  -1\right)  ^{a}%
\sum\limits_{j_{1}+\cdots+j_{r-1}=a}\binom{a}{j_{1},...,j_{r-1}}\frac{\left(
\left(  -1\right)  ^{n_{1}+\cdots+n_{r}+1}-1\right)  }{\left(  n_{1}%
-j_{1}\right)  !\cdots\left(  n_{r}+a+1\right)  !}\\
&  \quad\times b_{1}^{j_{1}}\cdots b_{r-1}^{j_{r-1}}b_{r}^{-a-1}E_{n_{1}%
-j_{1}}^{\left(  \alpha_{1}\right)  }\left(  y_{1}\right)  \cdots
E_{n_{r-1}-j_{r-1}}^{\left(  \alpha_{r-1}\right)  }\left(  y_{r-1}\right)
E_{n_{r}+a+1}^{\left(  \alpha_{r}\right)  }\left(  y_{r}\right)
\end{align*}
since $E_{n_{s}-j_{s}}^{\left(  \alpha_{s}\right)  }\left(  b_{s}%
-y_{s}\right)  =E_{n_{s}-j_{s}}^{\left(  \alpha_{s}\right)  }\left(
\alpha_{s}-y_{s}\right)  =\left(  -1\right)  ^{n_{s}-j_{s}}E_{n_{s}-j_{s}%
}^{\left(  \alpha_{s}\right)  }\left(  y_{s}\right)  $ and $j_{1}%
+\cdots+j_{r-1}=a.$ Therefore, if $n_{1}+\cdots+n_{r}+1$ is even, then
\[
\widehat{I}_{n_{1},...,n_{r}}\left(  1;\alpha_{1}-2y_{1},\ldots,\alpha
_{r}-2y_{r};y_{1},...,y_{r}\right)  =0,
\]
and if $n_{1}+\cdots+n_{r}+1$ is odd, then
\begin{align*}
&  \widehat{I}_{n_{1},...,n_{r}}\left(  1;\alpha_{1}-2y_{1},\ldots,\alpha
_{r}-2y_{r};y_{1},...,y_{r}\right) \\
&  =-2\sum\limits_{a=0}^{n_{1}+\cdots+n_{r-1}}\left(  -1\right)  ^{a}%
\frac{\left(  \alpha_{r}-2y_{r}\right)  ^{-a-1}}{\left(  n_{r}+a+1\right)
!}E_{n_{r}+a+1}^{\left(  \alpha_{r}\right)  }\left(  y_{r}\right) \\
&  \quad\times\sum\limits_{j_{1}+\cdots+j_{r-1}=a}\binom{a}{j_{1},...,j_{r-1}%
}\prod\limits_{s=1}^{r-1}\frac{\left(  \alpha_{s}-2y_{s}\right)  ^{j_{s}}%
}{\left(  n_{s}-j_{s}\right)  !}E_{n_{s}-j_{s}}^{\left(  \alpha_{s}\right)
}\left(  y_{s}\right)  .
\end{align*}
For example, we have%
\[
\int\limits_{0}^{1}E_{2}^{(3)}\left(  7z-2\right)  E_{3}^{(1/2)}\left(
-\frac{3}{2}z+1\right)  E_{10}^{(5)}\left(  4z+1/2\right)  dz=0
\]
and%
\[
\frac{1}{2!10!}\int\limits_{0}^{1}E_{2}^{(3)}\left(  3z\right)  E_{10}%
^{(5)}\left(  -3z+4\right)  dz=\frac{2}{3}\sum\limits_{a=0}^{2}\frac
{E_{2-a}^{(3)}\left(  0\right)  }{\left(  2-a\right)  !}\frac{E_{11+a}%
^{(5)}\left(  4\right)  }{\left(  11+a\right)  !}.
\]

It is seen from the definition of the integral $\widehat{I}_{n_{1},...,n_{r}%
}(x;b;y)$ that the left-hand side of (\ref{31}) is invariant under
interchanging the order of the integrands. That is, for\textbf{\ }$r=2,$%
\begin{align}
&  \sum\limits_{a=0}^{n}\left(  -1\right)  ^{a}\binom{m+n+1}{n-a}b_{1}%
^{a}b_{2}^{-a-1}\nonumber\\
&  \quad\times\left(  E_{n-a}^{\left(  \gamma\right)  }\left(  b_{1}%
x+y_{1}\right)  E_{m+a+1}^{\left(  \beta\right)  }\left(  b_{2}x+y_{2}\right)
-E_{n-a}^{\left(  \gamma\right)  }\left(  y_{1}\right)  E_{m+a+1}^{\left(
\beta\right)  }\left(  y_{2}\right)  \right) \nonumber\\
&  =\sum\limits_{a=0}^{m}\left(  -1\right)  ^{a}\binom{m+n+1}{m-a}b_{2}%
^{a}b_{1}^{-a-1}\nonumber\\
&  \quad\times\left(  E_{m-a}^{\left(  \gamma\right)  }\left(  b_{2}%
x+y_{2}\right)  E_{n+a+1}^{\left(  \beta\right)  }\left(  b_{1}x+y_{1}\right)
-E_{m-a}^{\left(  \gamma\right)  }\left(  y_{2}\right)  E_{n+a+1}^{\left(
\beta\right)  }\left(  y_{1}\right)  \right)  . \label{25}%
\end{align}
So, we may investigate the reciprocity relation for sums of products of
higher-order Euler polynomials as follows: Let
\begin{align*}
T  &  :=\sum\limits_{a=0}^{n}\left(  -1\right)  ^{a}\binom{m+n+1}{n-a}%
b_{1}^{a}b_{2}^{-a-1}E_{n-a}^{\left(  \gamma\right)  }\left(  y_{1}\right)
E_{m+a+1}^{\left(  \beta\right)  }\left(  y_{2}\right) \\
&  -\sum\limits_{a=0}^{m}\left(  -1\right)  ^{a}\binom{m+n+1}{m-a}b_{2}%
^{a}b_{1}^{-a-1}E_{m-a}^{\left(  \gamma\right)  }\left(  y_{2}\right)
E_{n+a+1}^{\left(  \beta\right)  }\left(  y_{1}\right)  .
\end{align*}
We first rewrite this as%
\begin{align}
T  &  =\sum\limits_{a=0}^{n}\left(  -1\right)  ^{n-a}\binom{m+n+1}{a}%
b_{1}^{n-a}b_{2}^{a-n-1}E_{a}^{\left(  \gamma\right)  }\left(  y_{1}\right)
E_{m+n+1-a}^{\left(  \beta\right)  }\left(  y_{2}\right) \nonumber\\
&  \quad-\sum\limits_{a=0}^{m}\left(  -1\right)  ^{m-a}\binom{m+n+1}{a}%
b_{2}^{m-a}b_{1}^{a-m-1}E_{a}^{\left(  \gamma\right)  }\left(  y_{2}\right)
E_{m+n+1-a}^{\left(  \beta\right)  }\left(  y_{1}\right)  . \label{46}%
\end{align}
Without loss of generality we may assume that $n\geq m$; in this case we
separate the sum from $0$ to $m$ and $m+1$ to $n$ on the first summation in
(\ref{46}), and rewrite these as%
\begin{align*}
&  \sum\limits_{a=0}^{m}\left(  -1\right)  ^{n-a}\binom{m+n+1}{a}b_{1}%
^{n-a}b_{2}^{a-n-1}E_{a}^{\left(  \gamma\right)  }\left(  y_{1}\right)
E_{m+n+1-a}^{\left(  \beta\right)  }\left(  y_{2}\right) \\
&  =\sum\limits_{a=n+1}^{m+n+1}\left(  -1\right)  ^{m+1-a}\binom{m+n+1}%
{a}b_{1}^{a-m-1}b_{2}^{m-a}E_{m+n+1-a}^{\left(  \gamma\right)  }\left(
y_{1}\right)  E_{a}^{\left(  \beta\right)  }\left(  y_{2}\right)
\end{align*}
and
\begin{align*}
&  \sum\limits_{a=m+1}^{n}\left(  -1\right)  ^{n-a}\binom{m+n+1}{a}b_{1}%
^{n-a}b_{2}^{a-n-1}E_{a}^{\left(  \gamma\right)  }\left(  y_{1}\right)
E_{m+n+1-a}^{\left(  \beta\right)  }\left(  y_{2}\right) \\
&  =\sum\limits_{a=m+1}^{n}\left(  -1\right)  ^{m+1-a}\binom{m+n+1}{a}%
b_{1}^{a-m-1}b_{2}^{m-a}E_{m+n+1-a}^{\left(  \gamma\right)  }\left(
y_{1}\right)  E_{a}^{\left(  \beta\right)  }\left(  y_{2}\right)  .
\end{align*}
Thus, we have
\begin{align}
T  &  =\frac{1}{b_{1}^{m+1}b_{2}^{n+1}}\sum\limits_{a=0}^{m+n+1}\left(
-1\right)  ^{m+1-a}\binom{m+n+1}{a}\nonumber\\
&  \quad\times b_{1}^{a}b_{2}^{m+n+1-a}E_{m+n+1-a}^{\left(  \gamma\right)
}\left(  y_{1}\right)  E_{a}^{\left(  \beta\right)  }\left(  y_{2}\right)  .
\label{47}%
\end{align}
Combining (\ref{25}) and (\ref{47}) gives the reciprocity relation for sums of
products of higher-order Euler polynomials.

\begin{corollary}%
\begin{align}
&  \sum\limits_{a=0}^{n}\left(  -1\right)  ^{a}\binom{m+n+1}{n-a}b_{1}%
^{a}b_{2}^{-a-1}E_{n-a}^{\left(  \gamma\right)  }\left(  b_{1}x+y_{1}\right)
E_{m+a+1}^{\left(  \beta\right)  }\left(  b_{2}x+y_{2}\right) \nonumber\\
&  \quad-\sum\limits_{a=0}^{m}\left(  -1\right)  ^{a}\binom{m+n+1}{m-a}%
b_{2}^{a}b_{1}^{-a-1}E_{m-a}^{\left(  \gamma\right)  }\left(  b_{2}%
x+y_{2}\right)  E_{n+a+1}^{\left(  \beta\right)  }\left(  b_{1}x+y_{1}\right)
\nonumber\\
&  =\frac{1}{b_{1}^{m+1}b_{2}^{n+1}}\sum\limits_{a=0}^{m+n+1}\left(
-1\right)  ^{m+1-a}\binom{m+n+1}{a}b_{1}^{a}b_{2}^{m+n+1-a}E_{m+n+1-a}%
^{\left(  \gamma\right)  }\left(  y_{1}\right)  E_{a}^{\left(  \beta\right)
}\left(  y_{2}\right)  . \label{45}%
\end{align}
In particular for $y_{1}=\gamma/2,$ $y_{2}=\beta/2$ and even $\left(
m+n\right)  ,$ the right-hand side of (\ref{45}) vanishes.
\end{corollary}

\begin{remark}
Beginning from the left-hand side of (\ref{45}) and using the arguments in the
proof of (\ref{47}), the right-hand side of (\ref{45}) turns into%
\begin{align*}
&  \frac{1}{b_{1}^{m+1}b_{2}^{n+1}}\sum\limits_{a=0}^{m+n+1}\left(  -1\right)
^{m+1-a}\binom{m+n+1}{a}\\
&  \times b_{1}^{a}b_{2}^{m+n+1-a}E_{m+n+1-a}^{\left(  \gamma\right)  }\left(
b_{1}x+y_{1}\right)  E_{a}^{\left(  \beta\right)  }\left(  b_{2}%
x+y_{2}\right)  .
\end{align*}
So it follows that for all $x$,%
\begin{align*}
&  \sum\limits_{a=0}^{m+n+1}\left(  -1\right)  ^{a}\binom{m+n+1}{a}b_{1}%
^{a}b_{2}^{m+n+1-a}E_{m+n+1-a}^{\left(  \gamma\right)  }\left(  b_{1}%
x+y_{1}\right)  E_{a}^{\left(  \beta\right)  }\left(  b_{2}x+y_{2}\right) \\
&  =\sum\limits_{a=0}^{m+n+1}\left(  -1\right)  ^{a}\binom{m+n+1}{a}b_{1}%
^{a}b_{2}^{m+n+1-a}E_{m+n+1-a}^{\left(  \gamma\right)  }\left(  y_{1}\right)
E_{a}^{\left(  \beta\right)  }\left(  y_{2}\right)  .
\end{align*}

\end{remark}

$\bullet$ Let $b_{1}=b_{2}=1$ in (\ref{45}). Then the right-hand side becomes,
with the use of (\ref{2}),%
\[
\left(  -1\right)  ^{n}T=\sum\limits_{a=0}^{m+n+1}\binom{m+n+1}{a}%
E_{m+n+1-a}^{\left(  \gamma\right)  }\left(  \gamma-y_{1}\right)
E_{a}^{\left(  \beta\right)  }\left(  y_{2}\right)  .
\]
Now using (\ref{41})\textbf{\ }by taking $x=y_{2}$ and $y=\gamma-y_{1},$
(\ref{45}) reduces to
\begin{align}
&  \sum\limits_{a=0}^{n}\left(  -1\right)  ^{a}\binom{m+n+1}{n-a}%
E_{n-a}^{\left(  \gamma\right)  }\left(  x+y_{1}\right)  E_{m+a+1}^{\left(
\beta\right)  }\left(  x+y_{2}\right) \nonumber\\
&  -\sum\limits_{a=0}^{m}\left(  -1\right)  ^{a}\binom{m+n+1}{m-a}%
E_{m-a}^{\left(  \gamma\right)  }\left(  x+y_{2}\right)  E_{n+a+1}^{\left(
\beta\right)  }\left(  x+y_{1}\right) \nonumber\\
&  =\left(  -1\right)  ^{n}2\left(  y_{2}-y_{1}+\gamma-1\right)
E_{m+n+1}^{\left(  \gamma+\beta-1\right)  }\left(  y_{2}-y_{1}+\gamma-1\right)
\nonumber\\
&  \quad-\left(  -1\right)  ^{n}2E_{m+n+2}^{\left(  \gamma+\beta-1\right)
}\left(  y_{2}-y_{1}+\gamma-1\right)  . \label{48}%
\end{align}

$\bullet$ Setting\textbf{\ }$b_{1}=1,$\ $b_{2}=-1$ and using (\ref{41}),
\textbf{(}\ref{45}\textbf{) }becomes\textbf{\ }%
\begin{align*}
&  \sum\limits_{a=0}^{n}\binom{m+n+1}{n-a}E_{n-a}^{\left(  \gamma\right)
}\left(  x+y_{1}\right)  E_{m+a+1}^{\left(  \beta\right)  }\left(
y_{2}-x\right) \\
&  +\sum\limits_{a=0}^{m}\binom{m+n+1}{m-a}E_{m-a}^{\left(  \gamma\right)
}\left(  y_{2}-x\right)  E_{n+a+1}^{\left(  \beta\right)  }\left(
x+y_{1}\right) \\
&  =2\left(  y_{2}+y_{1}-1\right)  E_{m+n+1}^{\left(  \gamma+\beta-1\right)
}\left(  y_{2}+y_{1}-1\right)  -2E_{m+n+2}^{\left(  \gamma+\beta-1\right)
}\left(  y_{2}+y_{1}-1\right)  .
\end{align*}

$\bullet$ Set $\beta=\gamma=1,$\textbf{\ }$b_{1}=2$ and\textbf{\ }$b_{2}=-1$
in \textbf{(}\ref{45}\textbf{). }In view of \cite[Theorem 6]{cz}, (\ref{45})
becomes%
\begin{align*}
&  \sum\limits_{a=0}^{n}\binom{m+n+1}{n-a}2^{m+1+a}E_{n-a}\left(
2x+y_{1}\right)  E_{m+a+1}\left(  -x+y_{2}\right) \\
&  +\sum\limits_{a=0}^{m}\binom{m+n+1}{m-a}2^{m-a}E_{m-a}\left(
-x+y_{2}\right)  E_{n+a+1}\left(  2x+y_{1}\right) \\
&  =\sum\limits_{a=0}^{m+n+1}\binom{m+n+1}{a}2^{a}E_{a}\left(  y_{2}\right)
E_{m+n+1-a}\left(  y_{1}\right) \\
&  =E_{m+n+1}(2y_{2}+y_{1})+2^{m+n+1}E_{m+n+1}\left(  \frac{2y_{2}+y_{1}}%
{2}\right) \\
&  \quad-2^{m+n+1}E_{m+n+1}\left(  \frac{2y_{2}+y_{1}+1}{2}\right)  .
\end{align*}

$\bullet$ Let $\gamma=\beta=1$ and $y_{1}=y_{2}=0$ in (\ref{45}). Then,
\begin{align}
&  \sum\limits_{a=0}^{n}\left(  -1\right)  ^{a}\binom{m+n+1}{n-a}b_{1}%
^{a}b_{2}^{-a-1}E_{n-a}\left(  b_{1}x\right)  E_{m+a+1}\left(  b_{2}x\right)
\nonumber\\
&  -\sum\limits_{a=0}^{m}\left(  -1\right)  ^{a}\binom{m+n+1}{m-a}b_{2}%
^{a}b_{1}^{-a-1}E_{m-a}\left(  b_{2}x\right)  E_{n+a+1}\left(  b_{1}x\right)
\nonumber\\
&  =\frac{1}{b_{1}^{m+1}b_{2}^{n+1}}\sum\limits_{a=0}^{m+n+1}\left(
-1\right)  ^{m+1-a}\binom{m+n+1}{a}b_{1}^{a}b_{2}^{m+n+1-a}E_{m+n+1-a}\left(
0\right)  E_{a}\left(  0\right)  . \label{47a}%
\end{align}
From the property $B_{2n+1}(0)=0$, $n\geq1$ and (\ref{49}) for $x=0,$%
\textbf{\ }we have\textbf{\ }$\left(  -1\right)  ^{a}E_{a}(0)=-E_{a}(0)$ for
$a>0$\textbf{. }Then\textbf{, }the right-hand side of (\ref{47a}) can be
written%
\begin{align}
T  &  =\frac{\left(  -1\right)  ^{m}}{b_{1}^{m+1}b_{2}^{n+1}}\sum
\limits_{a=0}^{m+n+1}\binom{m+n+1}{a}b_{1}^{a}b_{2}^{m+n+1-a}\nonumber\\
&  \quad\times E_{m+n+1-a}\left(  0\right)  E_{a}\left(  0\right)
-2\frac{\left(  -b_{2}\right)  ^{m}}{b_{1}^{m+1}}E_{m+n+1}\left(  0\right)  .
\label{47b}%
\end{align}

\begin{remark}
Kim and Son \cite{mj} proved the reciprocity formula for generalized Dedekind
sums $T_{r}(c,d)$ as%
\begin{equation}
cd^{r}T_{r}(c,d)+dc^{r}T_{r}(d,c)=-\frac{1}{2}%
{\displaystyle\sum\limits_{a=0}^{r}}
\binom{r}{a}d^{a-1}c^{r-1-a}\overline{E}_{a}(0)\overline{E}_{r-a}%
(0)+\overline{E}_{r+1}(0), \label{47c}%
\end{equation}
where $T_{r}(d,c)$ is defined by%
\[
T_{r}(c,d)=\sum\limits_{j=0}^{\left\vert d\right\vert -1}\left(  -1\right)
^{j}\overline{E}_{1}\left(  \frac{j}{d}\right)  \overline{E}_{r}\left(
\frac{cj}{d}\right)
\]
in which%
\begin{align*}
\overline{E}_{r}(x)  &  =E_{r}(x),\text{ }0\leq x<1,\\
\overline{E}_{r}(x+p)  &  =\left(  -1\right)  ^{p}\overline{E}_{r}(x),\text{
}p\in\mathbb{Z}\text{.}%
\end{align*}
It is seen from (\ref{47b}) and (\ref{47c}) that the reciprocity formula of
the generalized Dedekind sum $T_{r}(c,d)$ can be written in terms of the
reciprocity relation of Euler polynomials.
\end{remark}

Now, let us give the Laplace transform of $\overline{E}_{n}\left(  tu\right)
$ by applying (\ref{40}).

\begin{example}
Let $Re\left(  s\right)  >0$ and $|s/t|<\pi.$ Setting $f(u)=e^{-su}$ and
$\overline{E}_{n}\left(  tu\right)  $ instead of $E_{n_{r}}^{\left(
\alpha_{r}\right)  }\left(  u\right)  $ in\textbf{\ }(\ref{40}) gives%
\begin{align}
\frac{1}{n!}\int\limits_{0}^{x}e^{-su}\overline{E}_{n}\left(  tu\right)  du
&  =\sum\limits_{a=0}^{\mu}\frac{s^{a}t^{-a-1}}{\left(  n+a+1\right)
!}\left\{  e^{-sx}\overline{E}_{n+a+1}(tx)-\overline{E}_{n+a+1}(0)\right\}
\nonumber\\
&  +\left(  \frac{s}{t}\right)  ^{\mu+1}\frac{1}{\left(  n+\mu+1\right)
!}\int\limits_{0}^{x}e^{-su}\overline{E}_{n+\mu+1}\left(  tu\right)  du.
\label{37}%
\end{align}
Since the function $\overline{E}_{m}\left(  u\right)  =\left(  -1\right)
^{\left[  u\right]  }E\left(  u-\left[  u\right]  \right)  $ is bounded, the
integrals in (\ref{37}) converge absolutely and $e^{-sx}\overline{E}%
_{n+a+1}(tx)$ tends to $0$ as $x\rightarrow\infty.$ Then, letting
$x\rightarrow\infty,$ we have%
\begin{align}
\frac{1}{n!}\int\limits_{0}^{\infty}e^{-su}\overline{E}_{n}\left(  tu\right)
du  &  =-\frac{t^{n}}{s^{n+1}}\sum\limits_{a=0}^{\mu}\frac{E_{n+a+1}%
(0)}{\left(  n+a+1\right)  !}\frac{s^{n+a+1}}{t^{n+a+1}}\nonumber\\
&  +\left(  \frac{s}{t}\right)  ^{\mu+1}\frac{1}{\left(  n+\mu+1\right)
!}\int\limits_{0}^{\infty}e^{-su}\overline{E}_{n+\mu+1}\left(  tu\right)  du.
\label{39}%
\end{align}
From (\ref{0}) the sum in (\ref{39}) converges absolutely for $|s/t|<\pi$ as
$\mu\rightarrow\infty.$ Also the sequence of the functions \textbf{(}%
in\textbf{ }$u$\textbf{)} $s^{\mu}\overline{E}_{\mu}\left(  tu\right)
/\mu!t^{\mu}$ converges uniformly to $0$ for $|s/t|<\pi$. Thus, letting
$\mu\rightarrow\infty$ and using (\ref{0}), we obtain the Laplace transform of
$\overline{E}_{n}\left(  tu\right)  $
\begin{align}
\frac{1}{n!}\int\limits_{0}^{\infty}e^{-su}\overline{E}_{n}\left(  tu\right)
du  &  =-\frac{t^{n}}{s^{n+1}}\sum\limits_{a=0}^{\infty}\frac{\overline
{E}_{n+a+1}(0)}{\left(  n+a+1\right)  !}\frac{s^{n+a+1}}{t^{n+a+1}}\nonumber\\
&  =\frac{t^{n}}{s^{n+1}}\left(  \sum\limits_{a=0}^{n}\frac{E_{a}(0)}{a!}%
\frac{s^{a}}{t^{a}}-\sum\limits_{a=0}^{\infty}\frac{E_{a}(0)}{a!}\frac{s^{a}%
}{t^{a}}\right) \nonumber\\
&  =\frac{1}{s}\sum\limits_{a=0}^{n}\frac{E_{a}(0)}{a!}\left(  \frac{t}%
{s}\right)  ^{n-a}-\frac{t^{n}}{s^{n+1}}\frac{2}{e^{s/t}+1}. \label{16}%
\end{align}
Note that for $t=1,$ (\ref{16}) coincides with \cite[ eq. (64)]{ml}.
Differentiating $m$ times both sides of (\ref{16}) with respect to $s$, we
have%
\[
\frac{\left(  -1\right)  ^{m}}{n!}\int\limits_{0}^{\infty}u^{m}e^{-su}%
\overline{E}_{n}\left(  tu\right)  du=\frac{d^{m}}{ds^{m}}\left(
\sum\limits_{a=0}^{n}\frac{E_{a}(0)}{a!}t^{n-a}s^{a-n-1}-\frac{t^{n}}{s^{n+1}%
}\frac{2}{e^{s/t}+1}\right)  .
\]
\medskip
\end{example}

\subsection{\textbf{\textit{Bernoulli and Euler polynomials}}}

\begin{theorem}
Let $b_{s}$ and $y_{s},$ $1\leq s\leq l+r$ be arbitrary real numbers with
$b_{s}\not =0.$ Let $N=n_{1}!\cdots n_{l}!m_{1}!\cdots m_{r}!$ and
\begin{align*}
J_{n_{1},...,m_{r}}(x;b;y)  &  =J_{n_{1},...,m_{r}}(x;b_{1},...,b_{l+r}%
;y_{1},...,y_{l+r})\\
&  =\frac{1}{N}\int\limits_{0}^{x}\prod\limits_{s=1}^{l}B_{n_{s}}^{(\gamma
_{s})}\left(  b_{s}z+y_{s}\right)  \prod\limits_{i=1}^{r}E_{m_{i}}^{(\beta
_{i})}\left(  b_{l+i}z+y_{l+i}\right)  dz,\\
D_{n_{1},...,m_{r}}(x;b;y)  &  =D_{n_{1},...,m_{r}}\left(  x;b_{1}%
,...,b_{l+r};y_{1},...,y_{l+r}\right) \\
&  =\frac{1}{N}\prod\limits_{s=1}^{l}B_{n_{s}}^{(\gamma_{s})}\left(
b_{s}x+y_{s}\right)  \prod\limits_{i=1}^{r}E_{m_{i}}^{(\beta_{i})}\left(
b_{l+i}x+y_{l+i}\right) \\
&  -\frac{1}{N}\prod\limits_{s=1}^{l}B_{n_{s}}^{(\gamma_{s})}\left(
y_{s}\right)  \prod\limits_{i=1}^{r}E_{m_{i}}^{(\beta_{i})}\left(
y_{l+i}\right)  .
\end{align*}
Then, for $\mu=n_{1}+\cdots+n_{l}+m_{1}+\cdots+m_{r-1},$
\begin{align}
J_{n_{1},...,m_{r}}(x;b;y)  &  =\sum\limits_{a=0}^{\mu}\left(  -1\right)
^{a}\sum\limits_{j_{1}+\cdots+j_{l+r-1}=a}\binom{a}{j_{1},...,j_{l+r-1}%
}\nonumber\\
&  \quad\times b_{1}^{j_{1}}\cdots b_{l+r-1}^{j_{l+r-1}}b_{l+r}^{-a-1}%
D_{n_{1}-j_{1},\ldots,m_{r-1}-j_{l+r-1},m_{r}+a+1}(x;b;y). \label{33}%
\end{align}

\end{theorem}

\begin{proof}
The proof can be obtained by using the arguments in the proof of Theorem
(\ref{th-i}).
\end{proof}

In order to obtain the reciprocity relation for sums of products of
higher-order Bernoulli and Euler polynomials, similar to $T,$ we define
\begin{align*}
T_{1}  &  :=\sum\limits_{a=0}^{n}\left(  -1\right)  ^{a}\binom{m+n+1}%
{n-a}b_{1}^{a}b_{2}^{-a-1}B_{n-a}^{(\gamma)}\left(  y_{1}\right)
E_{m+a+1}^{(\beta)}\left(  y_{2}\right) \\
&  -\sum\limits_{a=0}^{m}\left(  -1\right)  ^{a}\binom{m+n+1}{m-a}b_{2}%
^{a}b_{1}^{-a-1}E_{m-a}^{(\gamma)}\left(  y_{2}\right)  B_{n+a+1}^{(\beta
)}\left(  y_{1}\right)  .
\end{align*}
Similarly, we have%
\begin{align}
T_{1}  &  =\sum\limits_{a=0}^{n}\left(  -1\right)  ^{a}\binom{m+n+1}{n-a}%
b_{1}^{a}b_{2}^{-a-1}B_{n-a}^{(\gamma)}\left(  b_{1}x+y_{1}\right)
E_{m+a+1}^{(\beta)}\left(  b_{2}x+y_{2}\right) \nonumber\\
&  \ -\sum\limits_{a=0}^{m}\left(  -1\right)  ^{a}\binom{m+n+1}{m-a}b_{2}%
^{a}b_{1}^{-a-1}E_{m-a}^{(\gamma)}\left(  b_{2}x+y_{2}\right)  B_{n+a+1}%
^{(\beta)}\left(  b_{1}x+y_{1}\right) \nonumber\\
&  =\frac{1}{b_{1}^{m+1}b_{2}^{n+1}}\sum\limits_{a=0}^{m+n+1}\left(
-1\right)  ^{m+1-a}\binom{m+n+1}{a}\nonumber\\
&  \quad\times b_{1}^{a}b_{2}^{m+n+1-a}E_{a}^{\left(  \beta\right)  }\left(
y_{2}\right)  B_{m+n+1-a}^{\left(  \gamma\right)  }\left(  y_{1}\right)  .
\label{30}%
\end{align}
Notice that the right-hand side of (\ref{30}) vanishes for $y_{1}=\gamma/2,$
$y_{2}=\beta/2$ and even $m+n.$

$\bullet$ Setting $\beta=\gamma=1$\textbf{\ }and\textbf{\ }$b_{1}=2,$
$b_{2}=-1$ in \textbf{(}\ref{30}\textbf{), }we get%
\begin{align*}
2^{m+1}T_{1}  &  =-\sum\limits_{a=0}^{m+n+1}\binom{m+n+1}{a}2^{a}E_{a}\left(
y_{2}\right)  B_{m+n+1-a}\left(  y_{1}\right) \\
&  =-B_{m+n+1}(2y_{2}+y_{1})+2^{m+n-1}\left(  m+n+1\right)  E_{m+n}\left(
\frac{2y_{2}+y_{1}+1}{2}\right) \\
&  \quad-2^{m+n-1}\left(  m+n+1\right)  E_{m+n}\left(  \frac{2y_{2}+y_{1}}%
{2}\right)
\end{align*}
by \cite[Theorem 10]{cz}. After similar manipulations to $T$, we have for
$\gamma=\beta$ and $b_{2}=b_{1}=1$%
\[
\left(  -1\right)  ^{n}T_{1}=\sum\limits_{a=0}^{m+n+1}\binom{m+n+1}%
{a}B_{m+n+1-a}^{(\gamma)}\left(  \gamma-y_{1}\right)  E_{a}^{(\gamma)}\left(
y_{2}\right)  .
\]
In view of (\ref{34}) for $x=\gamma-y_{1},$ $y=y_{2},$ we get%
\[
T_{1}=\left(  -1\right)  ^{n}2^{m+n+1}B_{m+n+1}^{(\gamma)}\left(  \frac
{\gamma-y_{1}+y_{2}}{2}\right)  .
\]

$\bullet$ On the other hand, for $y_{1}=y_{2}=0$ and $\gamma=\beta=1$, $T_{1}$
can be written as
\begin{align*}
T_{1}  &  =\sum\limits_{a=0}^{n}\left(  -1\right)  ^{a}\binom{m+n+1}{n-a}%
b_{1}^{a}b_{2}^{-a-1}B_{n-a}\left(  0\right)  E_{m+a+1}\left(  0\right) \\
&  \quad-\sum\limits_{a=0}^{m}\left(  -1\right)  ^{a}\binom{m+n+1}{m-a}%
b_{2}^{a}b_{1}^{-a-1}E_{m-a}\left(  0\right)  B_{n+a+1}\left(  0\right) \\
&  =\frac{(-1)^{m+1}}{b_{1}^{m+1}b_{2}^{n+1}}\sum\limits_{a=0}^{m+n+1}\left(
-1\right)  ^{a}\binom{m+n+1}{a}b_{1}^{a}b_{2}^{m+n+1-a}B_{m+n+1-a}\left(
0\right)  E_{a}\left(  0\right)  .
\end{align*}
Using (\ref{49}) for $x=0,$ we get%
\begin{align*}
T_{1}  &  =\frac{(-1)^{m+1}}{b_{1}^{m+1}b_{2}^{n+1}}\sum\limits_{a=1}%
^{m+n+2}\left(  -1\right)  ^{a-1}\binom{m+n+1}{a-1}\\
&  \times b_{1}^{a-1}b_{2}^{m+n+2-a}B_{m+n+2-a}\frac{2}{a}\left(
1-2^{a}\right)  B_{a}.
\end{align*}
Therefore, we have
\begin{align}
&  \sum\limits_{a=0}^{n}\left(  -1\right)  ^{a}\binom{m+n+1}{n-a}b_{1}%
^{a}b_{2}^{-a-1}B_{n-a}E_{m+a+1}\left(  0\right) \nonumber\\
&  -\sum\limits_{a=0}^{m}\left(  -1\right)  ^{a}\binom{m+n+1}{m-a}b_{2}%
^{a}b_{1}^{-a-1}E_{m-a}\left(  0\right)  B_{n+a+1}\nonumber\\
&  =\frac{(-1)^{m}}{b_{1}^{m+2}b_{2}^{n+1}}\frac{2}{m+n+2}\sum\limits_{a=1}%
^{m+n+2}\left(  -1\right)  ^{a}\binom{m+n+2}{a}\nonumber\\
&  \quad\times b_{1}^{a}b_{2}^{m+n+2-a}\left(  1-2^{a}\right)  B_{m+n+2-a}%
B_{a}. \label{50}%
\end{align}

\begin{remark}
Observe that the sum on the right-hand side of (\ref{50}) is the reciprocity
formula for the Hardy--Berndt sums $s_{3,r}(c,d)$ and $s_{4,r}(c,d)$ given by
\cite{mc}
\begin{align}
&  \left(  r+1\right)  \left(  cd^{r}s_{3,r}\left(  c,d\right)  -2^{-2}%
d\left(  2c\right)  ^{r}s_{4,r}\left(  d,c\right)  \right) \nonumber\\
&  \ =2\sum\limits_{a=1}^{r+1}\binom{r+1}{a}\left(  -1\right)  ^{a}%
c^{a}d^{r+1-a}\left(  1-2^{a}\right)  B_{a}B_{r+1-a}, \label{51}%
\end{align}
where $d$ and $r$ are odd and
\[
s_{3,r}(c,d)=\sum\limits_{j=1}^{d-1}\left(  -1\right)  ^{j}\overline{B}%
_{r}\left(  \frac{cj}{d}\right)  ,\text{ }s_{4,r}(c,d)=-4\sum\limits_{j=1}%
^{d-1}\overline{B}_{r}\left(  \frac{cj}{2d}\right)  .\text{\ }%
\]
Thus, the reciprocity formulas given by (\ref{50}) and (\ref{51}) can be
associated as
\begin{align*}
&  \sum\limits_{a=0}^{n}\left(  -1\right)  ^{m-a}\binom{m+n+1}{n-a}%
b_{1}^{m+2+a}b_{2}^{n-a}B_{n-a}E_{m+a+1}\left(  0\right) \\
&  \quad-\sum\limits_{a=0}^{m}\left(  -1\right)  ^{m-a}\binom{m+n+1}{m-a}%
b_{2}^{n+1+a}b_{1}^{m+1-a}E_{m-a}\left(  0\right)  B_{n+a+1}\\
&  \ =b_{1}b_{2}^{r}s_{3,r}(b_{1},b_{2})-2^{-2}b_{2}(2b_{1})^{r}s_{4,r}%
(b_{2},b_{1})\\
&  \ =\frac{2}{r+1}\sum\limits_{a=1}^{r+1}\left(  -1\right)  ^{a}\binom
{r+1}{a}b_{1}^{a}b_{2}^{r+1-a}\left(  1-2^{a}\right)  B_{r+1-a}B_{a}
\end{align*}
for odd integers $r=\left(  m+n+1\right)  $ and $b_{2}.$
\end{remark}

From this relationship, (\ref{50}) can be evaluated for some special cases.
Since $s_{3,r}(d,1)=0$ and $s_{4,r}(d,1)=0,$ we have%
\begin{align*}
&  \sum\limits_{a=0}^{n}\left(  -1\right)  ^{a}\binom{m+n+1}{n-a}%
b^{-a-1}B_{n-a}E_{m+a+1}\left(  0\right) \\
&  \quad-\sum\limits_{a=0}^{m}\left(  -1\right)  ^{a}\binom{m+n+1}{m-a}%
b^{a}E_{m-a}\left(  0\right)  B_{n+a+1}\\
&  =(-1)^{m}b^{m}s_{3,m+n+1}(1,b)
\end{align*}
for odd integers $\left(  m+n+1\right)  $ and $b,$ and
\begin{align*}
&  \sum\limits_{a=0}^{n}\left(  -1\right)  ^{a}\binom{m+n+1}{n-a}b^{a}%
B_{n-a}E_{m+a+1}\left(  0\right) \\
&  \quad-\sum\limits_{a=0}^{m}\left(  -1\right)  ^{a}\binom{m+n+1}%
{m-a}b^{-a-1}E_{m-a}\left(  0\right)  B_{n+a+1}\\
&  =(-1)^{m+1}2^{m+n-1}b^{n-1}s_{4,m+n+1}(1,b)
\end{align*}
for odd integer $\left(  m+n+1\right)  .$

\end{document}